\documentclass[11pt]{amsart}

\theoremstyle{plain}
\newtheorem{theorem}                {Theorem}      [section]
\newtheorem*{theorem*}                {Theorem \ref{thm:appl}}
\newtheorem{proposition}  [theorem]  {Proposition}
\newtheorem{corollary}    [theorem]  {Corollary}

\theoremstyle{definition}

\newtheorem{remark}       [theorem]  {Remark}
\newtheorem{definition}   [theorem]  {Definition}

\DeclareMathOperator{\trace}{trace} 

\DeclareMathOperator{\Div}{div} 
 
\DeclareMathOperator{\ricci}{Ricci}

\numberwithin{equation}{section}

\begin{document}

\title[On biconservative surfaces in $3$-dimensional space forms]
{On biconservative surfaces in $3$-dimensional space forms}

\author{Dorel~Fetcu}
\author{Simona~Nistor}
\author{Cezar~Oniciuc}

\address{Department of Mathematics and Informatics\\
Gh. Asachi Technical University of Iasi\\
Bd. Carol I, 11 \\
700506 Iasi, Romania} \email{dfetcu@math.tuiasi.ro}

\address{Faculty of Mathematics\\ Al. I. Cuza University of Iasi\\
Bd. Carol I, 11 \\ 700506 Iasi, Romania} \email{simona.nistor@math.uaic.ro}

\address{Faculty of Mathematics\\ Al. I. Cuza University of Iasi\\
Bd. Carol I, 11 \\ 700506 Iasi, Romania} \email{oniciucc@uaic.ro}

\thanks{}

\subjclass[2010]{53A10, 53C42}

\keywords{Biconservative surfaces, minimal surfaces, real space forms}

\begin{abstract} We consider biconservative surfaces $\left(M^2,g\right)$ in a space form $N^3(c)$, with mean curvature function $f$ satisfying $f>0$ and $\nabla f\neq 0$ at any point, and determine a certain Riemannian metric $g_r$ on $M$ such that $\left(M^2,g_r\right)$ is a Ricci surface in $N^3(c)$. We also obtain an intrinsic characterization of these biconservative surfaces. 
\end{abstract}

\maketitle

\section{Introduction}

In the last few years, from the theory of {\it biharmonic submanifolds}, arised the study of {\it biconservative submanifolds} that imposed itself as a very promising and interesting research topic through papers like \cite{CMOP,Fu,LMO,MOR,MOR2}. A {\it biharmonic map} $\psi:(M,g)\rightarrow(N,h)$ between two Riemannian manifolds is a critical point
of the \textit{bienergy functional}
$$
E_2:C^{\infty}(M,N)\rightarrow\mathbb{R},\quad E_{2}(\psi)=\frac{1}{2}\int_{M}|\tau(\psi)|^{2}\ dv,
$$
where $\tau(\psi)$ is the tension field of $\psi$, and it is characterized by the vanishing of its {\it bitension field} $\tau_2(\psi)$. When $\psi:(M,g)\rightarrow (N,h)$ is a biharmonic isometric immersion, $M$ is called a biharmonic submanifold of $N$.

Now, if $\psi:M\rightarrow(N,h)$ is a fixed map, then $E_2$ can be thought as a functional on the set of all Riemannian metrics on $M$. This new functional's critical points are Riemannian metrics determined by the vanishing of the {\it stress-energy tensor} $S_2$. This tensor satisfies
$$
\Div S_2=-\langle\tau_2(\psi),d\psi\rangle.
$$ 
If $\Div S_2=0$ for a submanifold $M$ in $N$, then $M$ is called a biconservative submanifold and it is characterized by the fact that the tangent part of its bitension field vanishes.

In the case when the ambient space is a $3$-dimensional space form $N^3(c)$, while surfaces with constant mean curvature (CMC surfaces) are trivially biconservative, the study of non-CMC biconservative surfaces is not trivial. The explicit local equations of these surfaces were obtained in \cite{CMOP} and \cite{Fu}. Moreover, in \cite{CMOP} it is shown that the Gaussian curvature of a biconservative surface in a $3$-dimensional space form satisfies a certain equation that seems to be very similar with that used by G.~Ricci-Curbastro \cite{R} in $1895$ to characterize minimal surfaces in $\mathbb{R}^3$. As we will see in the following, we can use this property of biconservative surfaces to prove results similar to those in \cite{L}, \cite{MM}, or \cite{R}, in this context.

The paper is organized as follows. After a short section where we recall some notions and results on biconservative submanifolds, we show, in the third section,  that on a non-CMC biconservative surface $\left(M^2,g\right)$ in a space form $N^3(c)$ we can determine a new Riemannian metric $g_r$ such that $\left(M^2,g_r\right)$ is a Ricci surface in $N^3(c)$. Then, in the last section of the paper, we obtain an intrinsic characterization of non-CMC biconservative surfaces in a space form.

\noindent \textbf{Acknowledgments.} The authors wish to thank Sergiu Moroianu for useful comments and discussions. 

\section{Preliminaries}

As we have already seen, biharmonic maps $\psi:(M,g)\rightarrow(N,h)$, as suggested by J.~Eells and J.~H.~Sampson \cite{JEJS}, are the critical points of the bienergy functional. The corresponding Euler-Lagrange equation, obtained in \cite{GYJ}, is
$$
\tau_{2}(\psi)=-\Delta\tau(\psi)-\trace R^N(d\psi,\tau(\psi))d\psi=0,
$$
where $\tau_{2}(\psi)$ is the \textit{bitension field} of $\psi$, $\Delta=-\trace(\nabla^{\psi})^2 =-\trace(\nabla^{\psi}\nabla^{\psi}-\nabla^{\psi}_{\nabla})$ is the rough Laplacian defined on
sections of $\psi^{-1}(TN)$ and $R^N$ is the curvature tensor of $N$, given by $R^N(X,Y)Z=[\bar\nabla_X,\bar\nabla_Y]Z-\bar\nabla_{[X,Y]}Z$. 

The {\it stress-energy tensor} associated to a variational problem, described in \cite{H} by D.~Hilbert, is a symmetric $2$-covariant tensor $S$ conservative at critical points, i.e., $S$ satisfies $\Div S=0$ at these points. 

P.~Baird and J.~Eells \cite{BE} and A.~Sanini \cite{S} used such a tensor given by 
$$
S=\frac{1}{2}|d\psi|^2g-\psi^{\ast}h
$$ 
to study harmonic maps. It has been proved that $S$ satisfies the equation
$$
\Div S=-\langle\tau(\psi),d\psi\rangle,
$$
which implies that $\Div S$ vanishes when $\psi$ is harmonic. When $\psi:M\to N$ is an isometric immersion, $\tau(\psi)$ is normal and then $\Div S=0$ always holds in this case.

Consider now the stress-energy tensor $S_2$ of the bienergy. This tensor, that was studied for the first time in \cite{J} and then in papers like \cite{CMOP,Fu,LMO,MOR,MOR2}, is given by
\begin{align*}
S_2(X,Y)=&\frac{1}{2}|\tau(\psi)|^2\langle X,Y\rangle+\langle d\psi,\nabla\tau(\psi)\rangle\langle X,Y\rangle\\&-\langle d\psi(X),\nabla_Y\tau(\psi)\rangle-\langle d\psi(Y),\nabla_X\tau(\psi)\rangle
\end{align*}
and it satisfies
$$
\Div S_2=-\langle\tau_2(\psi),d\psi\rangle.
$$
If $\psi:M\to N$ is an isometric immersion, then we have $\Div S_2=-\tau_2(\psi)^{\top}$ and, therefore, $\Div S_2$ does not automatically vanish.

\begin{definition} A submanifold $\psi:M\rightarrow N$ of a Riemannian manifold $N$ is called a {\it biconservative submanifold} if $\Div S_2=0$, i.e., $\tau_2(\psi)^{\top}=0$.
\end{definition}

The biharmonic equation $\tau_2(\psi)=0$ of a submanifold $\psi:M\rightarrow N$ can be decomposed in its normal and tangent part (see \cite{BMO,O}). In the case of hypersurfaces $M$ in $N$, we get
$$
\Delta f-f|A|^2+f\ricci^N(\eta,\eta)=0
$$
and
$$
2A(\nabla f)+f\nabla f-2f(\ricci^N(\eta))^{\top}=0,
$$
where $\eta$ is the unit normal of $M$ in $N$, $A$ is the shape operator, $f=\trace A$ is the mean curvature function, and $(\ricci^N(\eta))^{\top}$ is the tangent component of the Ricci curvature of $N$ in the direction of $\eta$.

From this decomposition, it follows that a surface $\psi:M^2\to N^3(c)$ in a space form $N^3(c)$, i.e., a $3$-dimensional simply connected complete manifold with constant sectional curvature $c$, is biconservative if and only if
\begin{equation}\label{eq:bicons}
A(\nabla f)=-\frac{f}{2}\nabla f.
\end{equation}

It is then easy to see that any CMC surface in $N^3(c)$ is biconservative and, therefore, when studying biconservative surfaces in space forms we are interested in the non-CMC case. We should, however, mention that, in the general case, if $M$ is a biconservative surface in an $n$-dimensional Riemannian manifold $N$, then it has constant mean curvature if and only if the $(2,0)$-part of the quadratic form $Q$, defined on $M$ by $Q(X,Y)=\langle B(X,Y),H\rangle$, is holomorphic, where $B$ is the second fundamental form of $M$ in $N$ and $H$ is the mean curvature vector field (see \cite{LO,MOR2}).

We end this section recalling the following result on non-CMC biconservative surfaces in $N^3(c)$ that we will use later on.

\begin{theorem}[\cite{CMOP}]\label{thm:CMOP} Let $M^2$ be a non-CMC biconservative surface in a space form $N^3(c)$. There exists an open subset $U\subset M$ such that, on $U$, the Gaussian curvature $K$ of $M$ satisfies
\begin{equation}\label{eq:Gauss}
K=\det A+c=-\frac{3f^2}{4}+c
\end{equation}
and
\begin{equation}\label{cond:K}
(c-K)\Delta K-|\nabla K|^2-\frac{8}{3}K(c-K)^2=0,
\end{equation}
where $A$ is the shape operator of $M$ in $N$, $f=\trace A$ is the mean curvature function, and $\Delta$ is the Laplace-Beltrami operator on $M$.
\end{theorem}

\begin{remark} It is easy to see, from \eqref{eq:Gauss}, that the Gaussian curvature $K$ of a non-CMC biconservative surface in $N^3(c)$ satisfies $c-K>0$.
\end{remark}

\noindent {\bf Convention.} Henceforth, all surfaces are assumed to be connected and oriented.

\section{Biconservativity and minimality in space forms}

A Riemannian surface $\left(M^2,g\right)$ with Gaussian curvature $K$ is said to satisfy the {\it Ricci condition} if $c-K>0$ and the metric $(c-K)^{1/2}g$ is flat, where $c\in\mathbb{R}$ is a constant. In this case, $\left(M^2,g\right)$ is called a {\it Ricci surface}. G. Ricci-Curbastro \cite{R} proved that, when $c=0$, a surface satisfying the Ricci condition can be locally isometrically embedded in $\mathbb{R}^3$ as a minimal surface. Actually, there exists a one-parameter family of such embeddings. H. B. Lawson \cite[Theorem 8]{L} generalized this result by showing that the Ricci condition is an intrinsic characterization of minimal surfaces in space forms $N^3(c)$, with constant sectional curvature $c$ (see also~\cite{Sak}).

In the following, we will see that the Ricci condition, as stated above, is equivalent to an equation that looks very much like equation \eqref{cond:K}, satisfied by the Gaussian curvature of a non-CMC biconservative surface in a space form $N^3(c)$. Then, a natural question is whether there exists a simple way to transform surfaces satisfying \eqref{cond:K} in Ricci surfaces in $N^3(c)$. As it will turn out, the answer to this question is affirmative.

We will first briefly recall some known results in conformal geometry. Let $(M^2,g)$ be a Riemannian surface with Gaussian curvature $K$ and Laplacian $\Delta$. Consider a new Riemannian metric $\bar g=e^{2\varphi}g$ on $M$, where $\varphi\in C^{\infty}(M)$. If $\bar\Delta$ and $\bar K$ are the Laplacian and the Gaussian curvature, respectively, of $\bar g$, then we have (see \cite{B}):
\begin{equation}\label{eq:confLaplacian}
\bar\Delta=e^{-2\varphi}\Delta
\end{equation}
and
\begin{equation}\label{eq:confK}
\bar K=e^{-2\varphi}(K+\Delta\varphi).
\end{equation}

The following proposition points out some equivalent characterizations of Ricci surfaces.

\begin{proposition}\label{p:Ricci_original} Let $\left(M^2,g\right)$ be a Riemannian surface such that its Gaussian curvature $K$ satisfies $c-K>0$, where $c\in\mathbb{R}$ is a constant. Then, the following conditions are equivalent$:$
\begin{itemize}
\item[(i)] $K$ satisfies 
\begin{equation}\label{cond:Ricci}
(c-K)\Delta K-|\nabla K|^2-4K(c-K)^2=0;
\end{equation}

\item[(ii)] $K$ satisfies
\begin{equation}\label{eq:MM}
\Delta\log(c-K)+4K=0;
\end{equation}

\item[(iii)] the metric $(c-K)^{1/2}g$ is flat.
\end{itemize}
Moreover, if $c=0$, then we also have a fourth equivalent condition$:$
\begin{itemize}
\item[(iv)] the metric $(-K)g$ has constant Gaussian curvature equal to $1$.
\end{itemize}
\end{proposition}

\begin{proof} First, we easily get that
$$
\Delta\log(c-K)=\frac{(K-c)\Delta K+|\nabla K|^2}{(c-K)^2},
$$
which implies that $(i)$ and $(ii)$ are equivalent.

Next, in the same way as in \cite{MM}, we consider a family of Riemannian metrics on $M$ given by $g_r=(c-K)^rg$, where $r\in\mathbb{R}$ is a constant. From equation \eqref{eq:confK}, one obtains that the Gaussian curvature curvature $K_r$ of $g_r$ is given by
$$
K_r=(c-K)^{-r}\left( K+\frac{1}{2}\Delta\log(c-K)^r\right).
$$
If $(ii)$ holds then $K_r=(1-2r)(c-K)^{-r}K$ and, therefore, $(ii)$ implies $(iii)$ and $(iv)$. Conversely, it is easy to see, from the expression of $K_r$, that $(iii)$ implies $(ii)$ and also, if $c=0$, $(iv)$ implies $(ii)$.
\end{proof}

\begin{remark} Proposition \ref{p:Ricci_original} was first proved in the case when $c=0$ in \cite{MM}.
\end{remark}

Working exactly as in the proof of Proposition \ref{p:Ricci_original} we get our following result.

\begin{proposition}\label{p:Ricci} Let $\left(M^2,g\right)$ be a Riemannian surface such that its Gaussian curvature $K$ satisfies $c-K>0$, where $c\in\mathbb{R}$ is a constant. Then, the following conditions are equivalent$:$
\begin{itemize}
\item[(i)] $K$ satisfies equation \eqref{cond:K}$;$

\item[(ii)] $\Delta\log(c-K)+(8/3)K=0$$;$

\item[(iii)] the metric $(c-K)^{3/4}g$ is flat.
\end{itemize}
Moreover, if $c=0$, then we also have a fourth equivalent condition$:$
\begin{itemize}
\item[(iv)] the metric $(-K)g$ has constant Gaussian curvature equal to $1/3$.
\end{itemize}
\end{proposition}

Now, we can state our first main result.

\begin{theorem}\label{thm:r3} Let $\left(M^2,g\right)$ be a Riemannian surface with negative Gaussian curvature $K$ that satisfies
\begin{equation}\label{cond:Kr3}
K\Delta K+|\nabla K|^2+\frac{8}{3}K^3=0.
\end{equation}
Then $\left(M^2,(-K)^{1/2}g\right)$ is a Ricci surface in $\mathbb{R}^3$.
\end{theorem}

\begin{proof} From Proposition \ref{p:Ricci_original}, one can see that suffices to show that there exists a Riemannian metric on $M$, conformally equivalent to $g$, that satisfies \eqref{eq:MM}.

In order to find such a metric, let us consider again the metrics $g_r=(-K)^rg$, with $r\in\mathbb{R}$. From \eqref{eq:confK} and Proposition \ref{p:Ricci}, one obtains that the Gaussian curvature curvature $K_r$ of $g_r$ is given by
$$
K_r=(-K)^{-r}\left(K+\frac{1}{2}\Delta\log(-K)^r\right)=-\frac{3-4r}{3}(-K)^{1-r}.
$$

Assume that $3-4r>0$, i.e., $K_r<0$, and then, using equations \eqref{eq:confLaplacian} and \eqref{cond:Kr3} and Proposition~\ref{p:Ricci}, we can compute
\begin{align*}
\Delta_r\log(-K_r)=&\Delta_r\log\left(\frac{3-4r}{3}(-K)^{1-r}\right)=(1-r)\Delta_r\log(-K)\\
=&(1-r)(-K)^{-r}\Delta\log(-K)\\=&\frac{8(1-r)}{3}(-K)^{1-r},
\end{align*}
where $\Delta_r$ is the Laplacian of $g_r$. Now, equation \eqref{eq:MM} becomes
$$
0=\Delta_r\log(-K_r)+4K_r=\Delta_r\log(-K_r)-\frac{4(3-4r)}{3}(-K)^{1-r}=\frac{4(2r-1)}{3}(-K)^{1-r}
$$
and we get that $r=1/2$.

We have just proved that $\left(M^2,g_{1/2}=(-K)^{1/2}g\right)$ is a Ricci surface with Gaussian curvature $K_{1/2}=-(1/3)(-K)^{1/2}<0$.
\end{proof}

From Theorems \ref{thm:CMOP} and \ref{thm:r3}, one obtains the following corollary.

\begin{corollary} Let $\left(M^2,g\right)$ be a biconservative surface in $\mathbb{R}^3$, where $g$ is the induced metric on $M$. If $f(p)>0$ and $(\nabla f)(p)\neq 0$ at any point $p\in M$, where $f$ is the mean curvature function, then $\left(M^2,(-K)^{1/2}g\right)$ is a Ricci surface.
\end{corollary}

\begin{remark} In the same way as in Theorem \ref{thm:r3} one can show that if $\left(M^2,g\right)$ is a Ricci surface in $\mathbb{R}^3$ with Gaussian curvature $K$, then the Gaussian curvature of $\left(M^2,(-K)^{-1}g\right)$ satisfies equation \eqref{cond:Kr3}.
\end{remark}

Although the method used to prove Theorem \ref{thm:r3} does not work in the case of non-flat space forms, it is still possible to extend this result to the case of space forms, as shown by the following theorem.

\begin{theorem}\label{thm:Nc} Let $\left(M^2,g\right)$ be a biconservative surface in a space form $N^3(c)$ with induced metric $g$ and Gaussian curvature $K$. If $f(p)>0$ and $(\nabla f)(p)\neq 0$ at any point $p\in M$, where $f$ is the mean curvature function, then, on an open dense set, $\left(M^2,(c-K)^rg\right)$ is a Ricci surface in $N^3(c)$, where $r$ is a locally defined function that satisfies
$$
K+\Delta\left(\frac{1}{4}\log(c-K_r)+\frac{r}{2}\log(c-K)\right)=0,
$$
with the Gaussian curvature $K_r$ of $(c-K)^rg$ given by
$$
K_r=(c-K)^{-r}\left(\frac{3-4r}{3}K+\frac{1}{2}\log(c-K)\Delta r+(c-K)^{-1}g(\nabla r,\nabla K)\right).
$$
\end{theorem}

\begin{proof} Let $A$ be the shape operator of $M^2$ in $N^3(c)$ and then $f=\trace A$ is the mean curvature function. Working as in \cite{CMOP}, we define on $M$ a global orthonormal frame field $\{X_1,X_2\}$, where $X_1=(\nabla f)/|\nabla f|$.

In \cite{CMOP} it is proved that $X_2f=0$, which implies, using Theorem \ref{thm:CMOP}, that also $X_2K=0$. In the same paper it is shown that
$$
AX_1=-\frac{f}{2}X_1,\quad AX_2=\frac{3f}{2}X_2
$$
and
\begin{equation}\label{eq:nabla}
\nabla_{X_1}X_1=\nabla_{X_1}X_2=0,\quad\nabla_{X_2}X_1=-\frac{3X_1f}{4f}X_2,\quad\nabla_{X_2}X_2=\frac{3X_1f}{4f}X_1,
\end{equation}
where $\nabla$ is the induced connection on $M$.

Now, let us consider a family of Riemannian metrics $g_r=(c-K)^rg$ on $M$, this time $r$ being a function on $M$ such that $X_2r=0$. From the above formulas for the Levi-Civita connection $\nabla$, it easily follows that $[X_1,X_2](K)=0$ and $[X_1,X_2](r)=0$. Therefore, we also have $X_2(X_1K)=0$ and $X_2(X_1r)=0$.

\noindent From \eqref{eq:confK}, we have that the Gaussian curvature $K_r$ of $g_r$ is given by
$$
K_r=(c-K)^{-r}\left(K+\frac{1}{2}\Delta(r\log(c-K))\right),
$$ 
where $K$ is the Gaussian curvature of $g$, and, since, after a straightforward computation, also using Theorem \ref{thm:CMOP} and Proposition \ref{p:Ricci}, we have that
\begin{align}\label{eq:delta_inter}
\Delta(r\log(c-K))=&\sum_{i=1}^2\{-X_i(X_i(r\log(c-K)))+\nabla_{X_i}X_i(r\log(c-K))\}\\\nonumber
=&r\Delta\log(c-K)+\log(c-K)\Delta r-2g(\nabla r,\nabla\log(c-K))\\\nonumber
=&-\frac{8}{3}rK+\log(c-K)\Delta r+2(c-K)^{-1}g(\nabla r,\nabla K),
\end{align}
it follows that
\begin{equation}\label{eq:Kr}
K_r=(c-K)^{-r}\left(\frac{3-4r}{3}K+\frac{1}{2}\log(c-K)\Delta r+(c-K)^{-1}g(\nabla r,\nabla K)\right).
\end{equation}

Next, assume that $(c-K_r)(p)>0$ at any point $p\in M$ and consider a new Riemannian metric $\bar g$ on $M$ given by 
\begin{align*}
\bar g=&(c-K_r)^{1/2}g_r=(c-K_r)^{1/2}(c-K)^rg\\
=&e^{2\varphi}g.
\end{align*}
We ask the corresponding Gaussian curvature $\bar K$ to vanish. From the definition of $\bar g$, one obtains 
\begin{equation}\label{eq:varphi}
\varphi=\frac{1}{4}\log(c-K_r)+\frac{r}{2}\log(c-K).
\end{equation}
Equation \eqref{eq:confK} implies that
$$
\bar K=(c-K_r)^{-1/2}(c-K)^{-r}(K+\Delta\varphi)
$$
and then $\bar K=0$ becomes 
\begin{equation}\label{eq:Kvanishes}
K+\Delta\varphi=0.
\end{equation}

Using \eqref{eq:varphi}, \eqref{eq:delta_inter}, and \eqref{eq:Kr} we get
\begin{align*}
\Delta\varphi=&\frac{1}{4}\Delta\log(c-K_r)+\frac{1}{2}\Delta(r\log(c-K))\\
=&\frac{1}{4}(c-K_r)^{-2}((K_r-c)\Delta K_r+|\nabla K_r|^2)-\frac{4}{3}rK+\frac{1}{2}\log(c-K)\Delta r\\
&+(c-K)^{-1}g(\nabla r,\nabla K)\\
=&\frac{1}{4}(c-K_r)^{-2}\bigg\{(K_r-c)\Delta\bigg((c-K)^{-r}\bigg(\frac{3-4r}{3}K+\frac{1}{2}\log(c-K)\Delta r\\&+(c-K)^{-1}g(\nabla r,\nabla K)\bigg)\bigg)\\
&+\bigg|\nabla\bigg((c-K)^{-r}\bigg(\frac{3-4r}{3}K+\frac{1}{2}\log(c-K)\Delta r+(c-K)^{-1}g(\nabla r,\nabla K)\bigg)\bigg)\bigg|^2\bigg\}\\
&-\frac{4}{3}rK+\frac{1}{2}\log(c-K)\Delta r+(c-K)^{-1}g(\nabla r,\nabla K).
\end{align*}

We note that, since $\nabla K\neq 0$ at any point, the function $\log(c-K)$ cannot vanish on an open subset of $M$. Now, away from the points where $\log(c-K)=0$, using the above equation, \eqref{eq:nabla}, and \eqref{eq:Gauss}, equation \eqref{eq:Kvanishes} can be written as
\begin{equation}\label{eq:Delta}
\Delta^2r=F(r,X_1r,X_1(X_1r),X_1(X_1(X_1r))),
\end{equation}
where the coefficients in the expression $F$ in the right hand side are smooth functions on $M$ depending on $K$, $X_1(K)$, $X_1(X_1K)$, and $X_1(X_1(X_1K))$. 

Let us consider a point $p_0\in M$ and $\gamma=\gamma(u)$ an integral curve of $X_1$ with $\gamma(0)=p_0$. Let $\phi$ be the flow of $X_2$ and, in a neighborhood $U\subset M$ of $p_0$, define a local parametrization of $M$,
$$
X(u,s)=\phi_{\gamma(u)}(s)=\phi(\gamma(u),s).
$$
We have $X(u,0)=\phi_{\gamma(u)}(0)=\gamma(u)$,
$$
X_u(u,0)=\partial_u(u,0)=\gamma'(u)=X_1(\gamma(u))=X_1(u,0),
$$
and
$$
X_s(u,s)=\partial_s(u,s)=\phi'_{\gamma(u)}(s)=X_2(\phi_{\gamma(u)}(s))=X_2(u,s),
$$
for any $u$ and $s$. 

By hypothesis, we have $X_2r=0$, which means that $r=r(u)$ on $U$. Moreover, since $X_2(X_1r)=0$, it follows that $(X_1r)(u,s)=(X_1r)(u,0)=r'(u)$ on $U$. From the formulas of the Levi-Civita connection, it is easy to see that also $X_2(X_1(X_1r))=0$, $X_2(X_1(X_1(X_1r)))=0$, and $X_2(X_1(X_1(X_1(X_1r))))=0$, that implies $X_1(X_1r)=r''(u)$, $X_1(X_1(X_1r))=r'''(u)$, and $X_1(X_1(X_1(X_1r)))=r^{(iv)}(u)$, respectively. Moreover, the same formulas hold if we take $K$ instead of $r$. Therefore, on $U$, equation \eqref{eq:Delta} becomes
\begin{equation}\label{eq:Delta_u}
r^{(iv)}(u)=\tilde F(u,r(u),r'(u),r''(u),r'''(u)).
\end{equation}
The initial conditions follow from $(c-K_r)(p_0)>0$, i.e., from 
\begin{align*}
(c-K(0))^{-r(0)}\bigg(\frac{3-4r(0)}{3}K(0)+\frac{1}{2}\log(c-K(0))\bigg(-r''(0)+\frac{3f'(0)}{4f(0)}r'(0)\bigg)\\+(c-K(0))^{-1}r'(0)K'(0)\bigg)<c.
\end{align*}
It is easy to see that we can choose $r(0)$, $r'(0)$, and $r''(0)$ such that the above inequality is satisfied.

Now, from the ODE's theory, we know that equation \eqref{eq:Delta_u}, with the given initial conditions, has a unique solution, which means that there exists a flat Riemannian metric $\bar g=(c-K_r)^{1/2}g_r$ on $U$, and then we use \cite[Theorem~8]{L} to conclude that our surface $(M^2,g)$ can be locally conformally embedded in $N^3(c)$ as a minimal surface.
\end{proof}

\begin{remark} It is straightforward to verify that, when $c=0$, the only constant solution of equation \eqref{eq:Delta} is $r=1/2$.
\end{remark}

\section{An intrinsic characterization of biconservative surfaces in space forms}

While any of the equivalent conditions in Proposition \ref{p:Ricci_original} characterizes intrinsically minimal surfaces in 3-dimensional space forms $N^3(c)$ (see \cite[Theorem~8]{L}), the similar conditions in Proposition \ref{p:Ricci} alone fail to do the same in the case of biconservative surfaces. In this section, we will find the intrinsic necessary and sufficient conditions for a Riemannian surface to be locally embedded in $N^3(c)$ as a non-CMC biconservative surface.

We will first need the following theorem.

\begin{theorem}\label{thm:metric}
Let $\left(M^2,g\right)$ be a Riemannian surface with Gaussian curvature $K$ satisfying $(\nabla K)(p)\neq 0$ and $c-K(p)>0$ at any point $p\in M$, where $c\in\mathbb{R}$ is a constant. Let $X_1=(\nabla K)/|\nabla K|$ and $X_2\in C(TM)$ be two vector fields on $M$ such that $\{X_1(p),X_2(p)\}$ is a positively oriented orthonormal basis at any point $p\in M$. If level curves of $K$ are circles in $M$ with constant curvature 
$$
\kappa=\frac{3X_1K}{8(c-K)}=\frac{3|\nabla K|}{8(c-K)},
$$
then, for any point $p_0\in M$, there exists a parametrization $X=X(u,s)$ of $M$ in a neighborhood $U\subset M$ of $p_0$ positively oriented such that
\begin{itemize}
\item[(a)] the curve $u\to X(u,0)$ is an integral curve of $X_1$ with $X(0,0)=p_0$ and $s\to X(u,s)$ is an integral curve of $X_2$, for any $u$$;$

\item[(b)] $K(u,s)=(K\circ X)(u,s)=(K\circ X)(u,0)=K(u)$, for any $(u,s)$;

\item[(c)] for any pair $(u,s)$, we have
$$
g_{11}(u,s)=\frac{9}{64}\left(\frac{K'(u)}{c-K(u)}\right)^2s^2+1,\quad g_{12}(u,s)=-\frac{3K'(u)}{8(c-K(u))}s,\quad g_{22}(u,s)=1;
$$

\item[(d)] the Gaussian curvature $K=K(u)$ satisfies
$$
24(c-K)K''+33(K')^2+64K(c-K)^2=0;
$$

\item[(e)] $X_1=X_u-g_{12}X_s$, $X_2=X_s$, the Levi-Civita connection on $\left(M^2,g\right)$ is given by
$$
\nabla_{X_1}X_1=\nabla_{X_1}X_2=0,\quad \nabla_{X_2}X_2=-\frac{3X_1K}{8(c-K)}X_1,\quad \nabla_{X_2}X_1=\frac{3X_1K}{8(c-K)}X_2,
$$
and, therefore, the integral curves of $X_1$ are geodesics.
\end{itemize}
\end{theorem} 

\begin{proof} Let $p_0$ be a fixed point in $M$, $\gamma=\gamma(u)$ an integral curve of $X_1$ with $\gamma(0)=p_0$, and $\phi$ the flow of $X_2$. Consider again
$$
X(u,s)=\phi_{\gamma(u)}(s)=\phi(\gamma(u),s).
$$
As we have already seen, we have $X(u,0)=\phi_{\gamma(u)}(0)=\gamma(u)$,
$$
X_u(u,0)=\partial_u(u,0)=\gamma'(u)=X_1(\gamma(u))=X_1(u,0),
$$
and
$$
X_s(u,s)=\partial_s(u,s)=\phi'_{\gamma(u)}(s)=X_2(\phi_{\gamma(u)}(s))=X_2(u,s),
$$
for any $u$ and $s$.

Since $X_s(u,s)=X_2(u,s)$, it follows that $|X_s(u,s)|^2=1$, which means that 
\begin{equation}\label{eq:g1}
g_{22}(u,s)=1.
\end{equation}

We also have, for any $u$,
\begin{equation}\label{eq:g2}
g_{11}(u,0)=|X_u(u,0)|^2=1,\quad g_{12}(u,0)=g(X_u(u,0),X_s(u,0))=0.
\end{equation}

We will now find the expression of $X_1$ with respect to $X_u=\partial_u$ and $X_s=\partial_s$. We write $X_1=a\partial_u+b\partial_s$, where $a$ and $b$ are smooth functions. Using \eqref{eq:g1}, it follows that
$$
1=g(X_1,X_1)=a^2g_{11}+2abg_{12}+b^2g_{22}=a^2g_{11}+2abg_{12}+b^2,
$$
and
$$
0=g(X_1,X_2)=ag_{12}+bg_{22}=ag_{12}+b.
$$
From the second equation, one obtains $b=-ag_{12}$ and, replacing in the first one, we get $
1=a^2(g_{11}-g_{12}^2)$.
Let us denote $\sigma(u,s)=\sqrt{g_{11}-g_{12}^2}>0$ and then we have 
\begin{equation}\label{eq:x1}
X_1=\frac{1}{\sigma}\partial_u-\frac{g_{12}}{\sigma}\partial_s.
\end{equation}

Next, we note that, from the definition of $X_1$ and $X_2$, one obtains $X_2K=0$, i.e., the integral curves $s\to\phi_{\gamma(u)}(s)$ of $X_2$ are the level curves of $K$, that means that $s\to K(\phi_{\gamma(u)}(s))$ is a constant function. Also, identifying $K$ with $K\circ X$, we can write $K=K(u,s)$. Since $X_2K=0$, it follows that actually $K(u,s)=K(u,0)=K(u)$, for any pair $(u,s)$. The level curves $s\to\phi_{\gamma(u)}(s)$ of $K$ are parametrized by arc length and, by hypothesis, are circles with constant curvature $\kappa=3X_1K/(8(c-K))$, which means, also using \eqref{eq:x1} and the fact that $\kappa(u,s)=\kappa(u)$, that
\begin{align}\label{eq:inter_K}
X_1K=&\frac{8}{3}\kappa(c-K)=\frac{8}{3}\kappa(u)(c-K(u))\\\nonumber
=&\frac{1}{\sigma}K'=\frac{1}{\sigma(u,s)}K'(u),
\end{align}
which implies that $X_2(X_1K)=0$ and $\sigma(u,s)=\sigma(u)=1$, for any $u$ and $s$.

Let us consider a fixed $u$. As $\{X_2,-X_1\}$ is positively oriented, we have
\begin{align*}
\nabla_{\phi'_{\gamma(u)}(s)}\phi'_{\gamma(u)}(s)=&\nabla_{X_2}X_2=\kappa(-X_1)\\
=&\Gamma_{22}^1\partial_u+\Gamma_{22}^2\partial_s
\end{align*}
and then 
\begin{align*}
\kappa=&g\left(\nabla_{\phi'_{\gamma(u)}(s)}\phi'_{\gamma(u)}(s),-X_1\right)=g\left(\Gamma_{22}^1\partial_u+\Gamma_{22}^2\partial_s,-\partial_u+g_{12}\partial_s\right)\\
=&-\Gamma_{22}^1\left(g_{11}-g_{12}^2\right)=-\Gamma_{22}^1,
\end{align*}
where $\Gamma_{ij}^k$ are the Christoffel symbols.

Since, by the definition of $\Gamma_{22}^1$, using $1=\sigma^2=g_{11}-g_{12}^2$, we have
\begin{align*}
\Gamma_{22}^1=&\frac{1}{2}g^{11}\left(\frac{\partial g_{21}}{\partial s}+\frac{\partial g_{21}}{\partial s}-\frac{\partial g_{22}}{\partial u}\right)+\frac{1}{2}g^{12}\left(\frac{\partial g_{22}}{\partial s}+\frac{\partial g_{22}}{\partial s}-\frac{\partial g_{22}}{\partial s}\right)\\
=&g^{11}\frac{\partial g_{12}}{\partial s}=\frac{\partial g_{12}}{\partial s},
\end{align*}
one obtains $\kappa=-\partial g_{12}/\partial s$. From equation \eqref{eq:inter_K}, it follows that
$$
K'(u)=-\frac{8}{3}\frac{\partial g_{12}}{\partial s}(c-K(u)),
$$
which leads to
$$
\frac{\partial g_{12}}{\partial s}=-\frac{3K'(u)}{8(c-K(u))}=\frac{3}{8}(\log(c-K(u)))'
$$
and, therefore,
$$
g_{12}(u,s)=-\frac{3K'(u)}{8(c-K(u))}s+\alpha(u).
$$
But, from \eqref{eq:g2}, we know that $g_{12}(u,0)=0$, which implies that $\alpha(u)=0$, and we conclude that
\begin{equation}\label{eq:g3}
g_{12}(u,s)=-\frac{3K'(u)}{8(c-K(u))}s.
\end{equation}

Finally, since $1=\sigma^2=g_{11}-g_{12}^2$, we find
\begin{equation}\label{eq:g4}
g_{11}(u,s)=\frac{9}{64}\left(\frac{K'(u)}{c-K(u)}\right)^2s^2+1.
\end{equation}

Now, from the definition of Christoffel symbols and \eqref{eq:g1}, \eqref{eq:g3}, and \eqref{eq:g4}, we obtain, after a straightforward computation, that the Gauss equation of $(M^2,g)$
$$
K=-\frac{1}{g_{11}}\left\{(\Gamma_{12}^2)_u-(\Gamma_{11}^2)_s+\Gamma_{12}^1\Gamma_{11}^2+\Gamma_{12}^2\Gamma_{12}^2-\Gamma_{11}^2\Gamma_{22}^2-\Gamma_{11}^1\Gamma_{12}^2\right\}
$$
is equivalent to
\begin{equation}\label{eq:preK}
24(c-K)K''+33(K')^2+64K(c-K)^2=0.
\end{equation}

Again using the definition of Christoffel symbols and equations \eqref{eq:g1}, \eqref{eq:g3}, \eqref{eq:g4}, the expressions of $X_1$ and $X_2$, and equation \eqref{eq:preK}, we get the formulas for the Levi-Civita connection as given by the last item of the theorem.
\end{proof}

\begin{remark} It is easy to verify that, in the hypotheses of Theorem \ref{thm:metric}, equation
$$
24(c-K)K''+33(K')^2+64K(c-K)^2=0
$$
can be written as
$$
(c-K)\Delta K-|\nabla K|^2-\frac{8}{3}K(c-K)^2=0.
$$
\end{remark}

\begin{remark} Considering a change of coordinates $(u,s)\to\left(u,(c-K)^{3/8}s\right)=(u,v)$ in Theorem \ref{thm:metric}, we obtain, after a straightforward computation, a simpler expression 
$$
g=du^2+(c-K)^{-\frac{3}{4}}dv^2
$$
for the Riemannian metric on the surface. Moreover, if we consider a second change of coordinates $(u,v)\to\left(\int_{u_0}^u(c-K)^{3/8}du,v\right)=(\tilde u,\tilde v)$, then the metric $g$ can be written as 
$$
g=(c-K)^{-\frac{3}{4}}\left(d\tilde u^2+d\tilde v^2\right),
$$
where $K(\tilde u)=K(u(\tilde u))$, which means that $(\tilde u,\tilde v)$ are isothermal coordinates on the surface.
\end{remark}

The converse of Theorem \ref{thm:metric} is the following result, that can be proved by a straightforward computation.

\begin{theorem} Let $M^2$ be a surface and $c\in\mathbb{R}$ a constant. Consider a fixed point $p_0\in M$, a parametrization $X=X(u,s)$ of $M$ on a neighborhood $U\subset M$ of $p_0$ positively oriented, and $K=K(u)$ a function on $M$ such that $K'(u)>0$ and $c-K(u)>0$, for any $u$, and 
$$
24(c-K)K''+33(K')^2+64K(c-K)^2=0.
$$
Define a Riemannian metric $g=g_{11}du^2+2g_{12}duds+g_{22}ds^2$ on $U$ by
$$
g_{11}(u,s)=\frac{9}{64}\left(\frac{K'(u)}{c-K(u)}\right)^2s^2+1,\quad g_{12}(u,s)=-\frac{3K'(u)}{8(c-K(u))}s,\quad g_{22}(u,s)=1.
$$ 
Then $K$ is the Gaussian curvature of $g$ and its level curves, i.e., the curves $s\to X(u,s)$, are circles in $M$ with curvature $\kappa=3K'(u)/(8(c-K(u)))$.
\end{theorem}
 
We are now ready to prove the main result of this section, which provides an intrinsic characterization of non-CMC biconservative surfaces in a $3$-dimensional space form $N^3(c)$.

\begin{theorem}\label{thm:char} Let $(M^2,g)$ be a Riemannian surface and $c\in\mathbb{R}$ a constant. Then $M$ can be locally isometrically embedded in a space form $N^3(c)$ as a biconservative surface with positive mean curvature having the gradient different from zero at any point $p\in M$ if and only if the Gaussian curvature $K$ satisfies $c-K(p)>0$, $(\nabla K)(p)\neq 0$, and its level curves are circles in $M$ with curvature $\kappa=(3|\nabla K|)/(8(c-K))$.
\end{theorem}

\begin{proof} The direct implication was proved in \cite{CMOP} and \cite{Fu}.

To prove the converse, let us consider $X_1=(\nabla K)/|\nabla K|$ and $X_2\in C(TM)$ two vector fields such that $\{X_1(p),X_2(p)\}$ is a positively oriented orthonormal basis at any point $p\in M$. From Theorem \ref{thm:metric} we have seen that the Levi-Civita connection on $(M^2,g)$ is given by
$$
\nabla_{X_1}X_1=\nabla_{X_1}X_2=0,\quad \nabla_{X_2}X_2=-\frac{3X_1K}{8(c-K)}X_1,\quad \nabla_{X_2}X_1=\frac{3X_1K}{8(c-K)}X_2.
$$

Now, consider $f=(2/\sqrt{3})\sqrt{c-K}>0$ and, since $X_2K=0$, we easily get
$$
\nabla f=-\frac{X_1K}{\sqrt{3(c-K)}}=-\frac{\nabla K}{\sqrt{3(c-K)}}.
$$
Define $\tilde X_1=(\nabla f)/|\nabla f|=-X_1$ and $\tilde X_2=-X_2$ and then
$$
\nabla_{\tilde X_1}\tilde X_1=\nabla_{\tilde X_1}\tilde X_2=0
$$
and
$$
\nabla_{\tilde X_2}\tilde X_2=\nabla_{X_2}X_2=-\frac{3\tilde X_1K}{8(c-K)}\tilde X_1,\quad \nabla_{\tilde X_2}\tilde  X_1=\nabla_{X_2}X_1=\frac{3\tilde X_1K}{8(c-K)}\tilde X_2.
$$
Since $(\tilde X_1f)/f=-(\tilde X_1K)/(2(c-K))$, we obtain
$$
\nabla_{\tilde X_2}\tilde X_2=\frac{3\tilde X_1f}{4f}\tilde X_1\quad\textnormal{and}\quad\nabla_{\tilde X_2}\tilde  X_1=-\frac{3\tilde X_1f}{4f}\tilde X_2.
$$ 

Let us now consider a tensor field $A$ of type $(1,1)$ on $M$ defined by
$$
A\tilde X_1=-\frac{f}{2}\tilde X_1\quad\textnormal{and}\quad A\tilde X_2=\frac{3f}{2}\tilde X_2.
$$
It is straightforward to verify that $A$ satisfies the Gauss equation
$$
K=c+\det A
$$
and the Codazzi equation
$$
(\nabla_{\tilde X_1}A)\tilde X_2=(\nabla_{\tilde X_2}A)\tilde X_1,
$$
which means that $M$ can be locally isometrically embedded in $N^3(c)$ with $A$ its shape operator. Moreover, from the definition of $A$ it is easy to see that
$$
A(\nabla f)=-\frac{f}{2}\nabla f,
$$
and, from \eqref{eq:bicons}, it follows that $M$ is a biconservative surface in $N^3(c)$.
\end{proof}

\begin{remark} If the surface $M$ in Theorem \ref{thm:char} is simply connected, then the theorem holds globally, but, in this case, instead of a local isometric embedding we have a global isometric immersion.
\end{remark}

\begin{remark} Let $(M^2,g)$ be a simply connected Riemannian surface and $c\in\mathbb{R}$ a constant. If $M$ admits two biconservative isometric immersions in $N^3(c)$ such that their mean curvatures are positive with gradients different from zero at any point $p\in M$, then the two immersions differ by an isometry of $N^3(c)$.
\end{remark}

\end{document}